\documentclass[10pt, a4paper, fleqn, final]{article}
\usepackage[utf8]{inputenc}
\usepackage[T1]{fontenc}
\usepackage[english]{babel}
\usepackage{amsmath, amsthm, amssymb, mathtools}
\usepackage[backend=bibtex,maxnames=4,sortcites=true,sorting=nyt]{biblatex}
\usepackage{libertine}
\usepackage{dsfont}
\usepackage{hyperref, url}
\usepackage{csquotes}

\usepackage[obeyDraft]{todonotes}

\newtheorem{theo}{Theorem}

\newtheorem{prop}[theo]{Proposition}
\theoremstyle{definition}
\newtheorem{defi}[theo]{Definition}

\newtheorem{exam}[theo]{Example}

\allowdisplaybreaks[4]

\newcommand{\esg}[2]{#1\langle#2\rangle} 
\DeclarePairedDelimiter\abs{\lvert}{\rvert}

\title{A note on a Broken-cycle Theorem \\ for hypergraphs}
\author{Martin Trinks\thanks{Supported by European Social Fond grant 080940498.} \\
\small Faculty Mathematics / Sciences / Computer Science \\[-0.8ex]
\small Hochschule Mittweida --- University of Applied Sciences \\[-0.8ex]
\small Technikumplatz 17, 09648 Mittweida, Germany \\
\small \tt trinks@hs-mittweida.de}

\date{\today}

\begin{document}

\maketitle


\begin{abstract}
Whitney's Broken-cycle Theorem states the chromatic polynomial of a graph as a sum over special edge subsets. We give a definition of cycles in hypergraphs that preserves the statement of the theorem there.
\end{abstract}

\section{Introduction}

The well-known Broken-cycle Theorem, originally given by Whitney \cite{whitney1932,whitney1931}, states the chromatic polynomial of a graph as a sum over edge subsets not including any broken cycles, where a broken cycle arises from the deletion of the maximal edge (with respect to a given order on the edge set) of a cycle.

While there are some definitions of cycles in hypergraphs \cite{jegou2009}, including the most common one due to Berge \cite[Section 5.1]{berge1989}, none of these definitions admits a straightforward generalization of broken cycles in hypergraphs such that the Broken-cycle Theorem is valid in this more general case.

We give a novel definition of cycles in hypergraphs that preserves the statement of the Broken-cycle Theorem. Therein cycles are minimal subgraphs such that the deletion of an edge does not increase the number of connected components. 

Furthermore, we extend the result to graph polynomials similar to the chromatic polynomial and to regarding a subset of the broken cycles. Both generalizations are already used in the case of graphs \cite[Subsection 3.2.1]{trinks2012c}.

\begin{defi}
\label{defi:hypergraph}
A \emph{hypergraph} $G = (V, E)$ is an ordered pair of a finite set of vertices, the vertex set $V$, and a finite multiset of (hyper)edges, the edge set $E$, such that each edge is a non-empty subset of the vertex set, i.e.\ $e \subseteq V$ for all $e \in E$.
\end{defi}

Consequently, a graph is a hypergraph $G = (V, E)$, where each edge is a set of at most two vertices: $\abs{e}\leq 2$ for all $e \in E$.

For a hypergraph $G = (V, E)$ we use the standard notations known from graphs, in particular the following ones: A hypergraph $G' = (V', E')$ is a \emph{subgraph} of $G$, if $V' \subseteq V$ and $E' \subseteq E$. A hypergraph $\esg{G}{A} = (V, A)$ for an edge subset $A \subseteq E$ is a \emph{spanning subgraph}. Furthermore, we denote by $k(G)$ the number of connected components and by $G_{-e}$ the graph arising from $G$ by deleting $e$.

\todo[inline]{Erklärung warum $\delta$?}

\begin{defi}
\label{defi:cyclic_hypergraph} Let $G = (V, E)$ be a hypergraph. $G$ is \emph{$\delta$-cyclic}, if it has a subgraph $G' = (V', E')$ including at least one edge such that for each edge $e \in E'$ it holds
\begin{align}
k(G') = k(G'_{-e}).
\end{align}
\end{defi}

\begin{defi}
\label{defi:cycle}
Let $G = (V, E)$ be a hypergraph. $G$ is a \emph{$\delta$-cycle}, if it is $\delta$-cyclic and has no proper $\delta$-cyclic subgraph.
\end{defi}

Therefore, in the case of graphs the definitions of $\delta$-cycles equals the usual definition of cycles (regarding a single loop and parallel edges also as cycles). 

\begin{exam}
Consider the hypergraph $G = (V, E)$ with $V = \{1, 2, 3, 4, 5\}$ and 
$E = \{ \{1, 3\}, \{1, 2, 3\}, \{1, 4, 5\}, \{3, 4, 5\} \}$. $G$ is $\delta$-cyclic but not a $\delta$-cycle, because the deletion of edge $\{1, 2, 3\}$ renders vertex $2$ isolated. The subgraph arising from deleting the edge $\{1, 2, 3\}$ and the vertex $2$ is the only $\delta$-cycle of $G$. $G$ itself is a cycle due to the definition of Berge \cite[Section 5.1]{berge1989}.
\end{exam}

\todo[inline]{Delete the last sentence?}

We consider hypergraphs $G = (V, E)$ with a linear order $<$ on the edge set $E$. This linear order can be represented by a bijection $\beta \colon E \rightarrow \{1, \ldots, \abs{E}\}$ for all $e, f \in E$ with
\begin{align}
e < f \Leftrightarrow \beta(e) < \beta(f).
\end{align}

\begin{defi}
\label{defi:broken_cycle}
Let $G = (V, E)$ be a hypergraph with a linear order $<$ on the edge set $E$. Let $C = (V_C, E_C) \subseteq G$ be a $\delta$-cycle and $e \in E_C$ the maximal edge of $C$ with respect to $<$. Then $E_C \setminus \{e\}$ is a \emph{broken cycle} in $G$ with respect to $<$. The \emph{set of all broken cycles} of $G$ with respect to $<$ is denoted by $\mathcal{B}(G, <)$.
\end{defi}

\begin{defi}
Let $G = (V, E)$ be a hypergraph. A \emph{$k$-coloring} of $G$ is a function $\phi \colon V \rightarrow \{1, \ldots, k\}$. A $k$-coloring is \emph{proper}, if for each edge not all vertices are mapped to the same element, i.e.
\begin{align}
\nexists e \in E \,  \exists c \in \{1, \ldots, k\} \, \forall v \in e \colon \phi(v) = c.
\end{align} 
\end{defi}

\begin{defi}
Let $G = (V, E)$ be a hypergraph. The \emph{chromatic polynomial} $\chi(G, x)$ equals (for $x \in \mathds{N}$) the number of proper $x$-colorings.
\end{defi}

The chromatic polynomial of a hypergraph satisfies the same edge subset expansion that is valid in the case of graphs \cites[Theorem 2.21]{dong2005}[Section 2]{whitney1931}.

\todo[inline]{``Ausführlicher'', Referenz? Beweis? "To show this, the same proofs can be applied."?}

\begin{prop}[Proposition 1.1 in \cite{dohmen1995}]
Let $G = (V, E)$ be a hypergraph. The chromatic polynomial $\chi(G, x)$ satisfies
\begin{align}
\chi(G, x) = \sum_{A \subseteq E}{(-1)^{\abs{A}} x^{k(\esg{G}{A})}}. \label{eq:prop_cp_expansion}
\end{align}
\end{prop}

\section{A Broken-cycle Theorem for hypergraphs}

\begin{theo}
\label{theo:broken-cycle_theorem}
Let $G = (V, E)$ be a hypergraph with a linear order $<$ on the edge set $E$. The chromatic polynomial $\chi(G, x)$ satisfies
\begin{align}
\chi(G, x)
&= \sum_{\substack{A \subseteq E \\ \forall B \in \mathcal{B}(G, <) \colon B \nsubseteq A}}{(-1)^{\abs{A}} x^{k(\esg{G}{A})}}.
\end{align}
\end{theo}

\begin{proof}
Assume that $E = \{e_1, \ldots, e_{\abs{E}}\}$ such that $e_1 < \cdots < e_{\abs{E}}$. For each broken cycle $B \in \mathcal{B}(G, <)$ we denote by $e(B)$ the minimal edge closing the broken cycle $B$, i.e.\ 
\begin{align*}
e(B) = \min{\{ e \in E \mid B \cup \{e\} \text{ is the edge set of a } \delta\text{-cycle in } G\}}.
\end{align*}
We partition the edge subsets $A \subseteq E$ into blocks $E_i$ (some of them may be empty) as follows: $A \in E_0$ if $A$ does not include any broken cycles, and, otherwise, $A \in E_i$ if $e_i$ is the minimal edge closing a broken cycle included in	 $A$, i.e.\ $A \in E_i$ if $e_i = \min{\{ e(B) \mid B \in \mathcal{B}(G, <) \wedge B \subseteq A \}}$.

We claim that for each $i > 0$ and each $A \subseteq E$ with $e_i \notin A$ it holds
\begin{align*}
A \in E_i \Leftrightarrow A \cup \{e_i\} \in E_i.
\end{align*}

Proof of the first direction ($\Rightarrow$): We have $A \in E_i$ and assume that $A \cup \{e_i\} \in E_j$ with $i \neq j$, i.e.\ $e_j$ is the minimal edge closing a broken cycle in $A \cup \{e_i\}$. Because every broken cycle in $A$ is also a broken cycle in $A \cup \{e_i\}$, there is also a broken cycle closed by $e_i$ in $A \cup \{e_i\}$, and hence $e_j < e_i$. But there is no broken cycle closed by $e_j$ in $A$, otherwise $A \in E_j$, and therefore $e_i$ must be an edge in each broken cycle closed by $e_j$ in $A \cup \{e_i\}$. Consequently, as $e_j$ is greater than every edge of the broken cycle closed by it, $e_i < e_j$, which gives a contradiction. It follows $A \cup \{e_i\} \in E_i$.


Proof of the second direction ($\Leftarrow$): We have $A \cup \{e_i\} \in E_i$, i.e.\ $e_i$ is the minimal edge closing some broken cycle in $A \cup \{e_i\}$, and this broken cycle is also in $A$. Because every broken cycle in $A$ is also in $A \cup \{e_i\}$, $e_i$ is the minimal edge closing some broken cycle in $A$, consequently $A \in E_i$.


For such $A$ ($A \in E_i$ for $i > 0$) it follows that $e_i$ is an edge of a $\delta$-cycle in $\esg{G}{A \cup \{e_i\}}$ and from the definition of $\delta$-cycles it follows that $k(\esg{G}{A}) = k(\esg{G}{A \cup \{e_i\}})$.
Hence, for each block $E_i \neq E_0$ ($i > 0$) it holds
\begin{align*}
\sum_{A \in E_i}{(-1)^{\abs{A}} x^{k(\esg{G}{A})}} 
&= \sum_{\substack{A \in E_i \\ e_i \notin A}}{(-1)^{\abs{A}} x^{k(\esg{G}{A})}} + \sum_{\substack{A \in E_i \\ e_i \in A}}{(-1)^{\abs{A}} x^{k(\esg{G}{A})}} = 0.
\end{align*}
As $E_0$ is the set of edge subsets not including any broken cycle $B \in \mathcal{B}(G, <)$, we have $E_0 = \{A \subseteq E \mid \forall B \in \mathcal{B}(G, <) \colon B \nsubseteq A\}$ and the statement follows via the edge subset expansion of the chromatic polynomial given in Equation \eqref{eq:prop_cp_expansion}:
\begin{align*}
\chi(G, x)
&= \sum_{A \subseteq E}{(-1)^{\abs{A}} x^{k(\esg{G}{A})}} \\
&= \sum_{\substack{A \subseteq E \\ A \in E_0}}{(-1)^{\abs{A}} x^{k(\esg{G}{A})}} \\
&= \sum_{\substack{A \subseteq E \\ \forall B \in \mathcal{B}(G, <) \colon B \nsubseteq A}}{(-1)^{\abs{A}} x^{k(\esg{G}{A})}}. \qedhere
\end{align*}
\end{proof}

In the case of graphs, the term $k(\esg{G}{A})$ can be simplified to $\abs{V} - \abs{A}$ in broken-cycle-free spanning subgraphs. For hypergraphs this is not possible, because edges can connect a different number of vertices.

\todo[inline]{If every $\delta$-cycle is also a cycle due to Berge, than ...}

\section{A further generalization}

\begin{theo}
\label{theo:generalized_broken-cycle_theorem}
Let $G = (V, E)$ be a hypergraph with a linear order $<$ on the edge set $E$, $\mathcal{B} \subseteq \mathcal{B}(G, <)$ a subset of the set of broken cycles of $G$, and $f(G, A)$ a function to an additive abelian group such that for all $A \subseteq E$ and all $e \in E \setminus A$ it holds
\begin{align}
\label{eq:theo_gen_bct_cond}
k(\esg{G}{A}) = k(\esg{G}{A \cup \{e\}}) \Rightarrow f(G, A) = - f(G, A \cup \{e\}).
\end{align}
Then
\begin{align}
\sum_{A \subseteq E}{f(G, A)}
&= \sum_{\substack{A \subseteq E \\ \forall B \in \mathcal{B} \colon B \nsubseteq A}}{f(G, A)}.  \label{eq:theo_gbct_cond}
\end{align}
\end{theo}

\begin{proof}
We use induction with respect to the cardinality of the set $\mathcal{B}$. For the basic step we assume that $\abs{\mathcal{B}} = 0$ and the statement holds obviously.

We assume as induction hypothesis that the statement holds for any set $\mathcal{B} \subseteq \mathcal{B}(G, <)$ with cardinality less than $k$ and consider now a set $\mathcal{B} \subseteq \mathcal{B}(G, <)$ with cardinality $k$.

For each broken cycle $B \in \mathcal{B}(G, <)$, we denote by $e(B)$ the maximal edge closing the broken cycle $B$, i.e.\ 
\begin{align*}
e(B) = \max{ \{e \in E \mid B \cup \{e\} \text{ is the edge set of a } \delta \text{-cycle in } G\} }.
\end{align*}
Let $B \in \mathcal{B}$ such that $\mathcal{B} = \mathcal{B}' \cup \{B\}$ and $e(B) \nless e(B')$ for all $B' \in \mathcal{B}'$.

In fact, we only have to show that the edge subsets that do include the broken cycle $B$, but do not include any broken cycle $B' \in \mathcal{B}'$, cancel each other. Let $\mathcal{A}$ be the set of such edge subsets, i.e.
\begin{align*}
\mathcal{A} = \bigcup_{\substack{A \subseteq E \\ \forall B' \in \mathcal{B}' \colon B' \nsubseteq A \\ B \subseteq A}}{\{A\}}.
\end{align*}
We claim that for each $A \in \mathcal{A}$ with $e(B) \notin A$ it holds
\begin{align*}
A \in \mathcal{A} \Leftrightarrow A \cup \{e(B)\} \in \mathcal{A}.
\end{align*}

Proof of the first direction ($\Rightarrow$): As $B \subseteq A$, obviously $B \subseteq A \cup \{e(B)\}$. Hence we have to show that there is no broken cycle $B' \in \mathcal{B}'$ with $B' \subseteq A \cup \{e(B)\}$. Assume there is such a broken cycle $B'$. Because $B' \nsubseteq A$, $e(B)$ must be an edge of $B'$, and consequently the maximal edge closing $B'$ must be greater than $e(B)$, $e(B) < e(B')$. This is a contradiction to the choice of $B$ such that $e(B) \nless e(B')$ for all $B' \in \mathcal{B}'$. Hence there is no such $B'$ and it follows $A \cup \{e(B)\} \in \mathcal{A}$.

Proof of the second direction ($\Leftarrow$): We have $A \cup \{e(B)\} \in \mathcal{A}$, i.e.\ $A \cup \{e(B)\}$ contains only the broken cycle $B$, which does not include $e(B)$ by definition. Therefore, $A$ contains the broken cycle $B$, but no other broken cycle, because otherwise this broken cycle would also be in $A \cup \{e(B)\}$. Consequently $A \in \mathcal{A}$.



Because $\abs{\mathcal{B}'} < k$ we can use the induction hypothesis and the statement follows by
\begin{align*}
\sum_{A \subseteq E}{f(G, A)}
&= \smashoperator[r]{\sum_{\substack{A \subseteq E \\ \forall B' \in \mathcal{B}' \colon B' \nsubseteq A}}}{f(G, A)} \\
&= \smashoperator[r]{\sum_{\substack{A \subseteq E \\ \forall B' \in \mathcal{B}' \colon B' \nsubseteq A \\ B \nsubseteq A}}}{f(G, A)} + \smashoperator[r]{\sum_{\substack{A \subseteq E \\ \forall B' \in \mathcal{B}' \colon B' \nsubseteq A \\ B \subseteq A}}}{f(G, A)} \\
&= \smashoperator[r]{\sum_{\substack{A \subseteq E \\ \forall B \in \mathcal{B} \colon B \nsubseteq A}}}{f(G, A)} + \smashoperator[r]{\sum_{\substack{A \subseteq E \\ \forall B' \in \mathcal{B}' \colon B' \nsubseteq A \\ B \subseteq A, e(B) \in A}}}{f(G, A)} + \smashoperator[r]{\sum_{\substack{A \subseteq E \\ \forall B' \in \mathcal{B}' \colon B' \nsubseteq A \\ B \subseteq A, e(B) \notin A}}}{f(G, A)} \\
&= \smashoperator[r]{\sum_{\substack{A \subseteq E \\ \forall B \in \mathcal{B} \colon B \nsubseteq A}}}{f(G, A)}. \qedhere
\end{align*}
\end{proof}

\printbibliography

\end{document}